\documentclass[12pt]{article}

\usepackage[T1]{fontenc}
\usepackage[utf8]{inputenc}
\usepackage[margin=2.54cm]{geometry}
\usepackage{microtype}
\usepackage{comment}
\usepackage[tbtags]{amsmath}
\usepackage{amssymb,amsthm,mathtools,mathrsfs}
\usepackage{thmtools}
\usepackage{thm-restate}
\numberwithin{equation}{section}
\allowdisplaybreaks

\usepackage{graphicx}
\usepackage{booktabs}
\usepackage{array,multirow}
\usepackage{float}
\usepackage{subcaption}
\usepackage[export]{adjustbox}
\usepackage{enumitem}
\usepackage{placeins}
\usepackage{sectsty}
\usepackage{xcolor}
\usepackage{hyperref}
\usepackage{orcidlink}
\definecolor{mycitegreen}{HTML}{00A99D}
\hypersetup{
	colorlinks=true,
	linkcolor=blue,
	citecolor={red!90!black},
	urlcolor=mycitegreen
}
\usepackage[noabbrev]{cleveref}

\usepackage{listings}
\definecolor{codebg}{RGB}{248,248,248}
\definecolor{codeframe}{RGB}{220,220,220}
\definecolor{codecomment}{RGB}{0,96,96}
\definecolor{codestring}{RGB}{128,0,0}
\definecolor{codekeyword}{RGB}{40,40,180}

\lstdefinestyle{pythonstyle}{
	language=Python,
	basicstyle=\ttfamily\small,
	numbers=left,
	numberstyle=\scriptsize\color{gray},
	stepnumber=1,
	numbersep=8pt,
	breaklines=true,
	showstringspaces=false,
	columns=fullflexible,
	frame=single,
	rulecolor=\color{codeframe},
	backgroundcolor=\color{codebg},
	keywordstyle=\color{codekeyword},
	commentstyle=\itshape\color{codecomment},
	stringstyle=\color{codestring}
}
\sectionfont{\centering}

\theoremstyle{plain}
\newtheorem{theorem}{Theorem}[section]

\newtheorem{lemma}[theorem]{Lemma}
\newtheorem{proposition}[theorem]{Proposition}
\newtheorem{corollary}[theorem]{Corollary}

\theoremstyle{definition}
\newtheorem{definition}{Definition}[section]

\newtheorem{question}[theorem]{Question}

\theoremstyle{remark}

\crefname{theorem}{Theorem}{Theorems}
\crefname{lemma}{Lemma}{Lemmas}
\crefname{proposition}{Proposition}{Propositions}
\crefname{corollary}{Corollary}{Corollaries}
\crefname{conjecture}{Conjecture}{Conjectures}
\crefname{problem}{Problem}{Problems}
\crefname{claim}{Claim}{Claims}
\crefname{observation}{Observation}{Observations}
\crefname{fact}{Fact}{Facts}
\crefname{definition}{Definition}{Definitions}
\crefname{construction}{Construction}{Constructions}
\crefname{question}{Question}{Questions}
\crefname{setup}{Setup}{Setups}
\crefname{example}{Example}{Examples}
\crefname{remark}{Remark}{Remarks}
\crefname{section}{Section}{Sections}
\Crefname{section}{Section}{Sections}
\crefname{subsection}{Subsection}{Subsections}
\Crefname{subsection}{Subsection}{Subsections}
\crefname{subsubsection}{Subsubsection}{Subsubsections}
\Crefname{subsubsection}{Subsubsection}{Subsubsections}
\crefrangelabelformat{lemma}{#3#1#4--#5#2#6}

\makeatletter
\newcommand{\labelinthm}[1]{%
	\label{temp#1}%
	\protected@write\@auxout{}{%
		\string\newlabel{#1}{{\emph{\ref{temp#1}}}{\thepage}{\emph{\ref{temp#1}}}{temp#1}{}}%
	}%
}
\makeatother

\title{Ramsey-Tur\'an Type Problem for Perfect Transitive Triangle Tilings in Digraphs}
\author{
	Zhimin Wang~\orcidlink{0009-0009-6772-6061}\thanks{School of Mathematics, Shandong University, Jinan 250100, China. Supported by the National Natural Science Foundation of China (No.~12571373).}
	\and
	Zhilan Wang\footnotemark[1]
	\and
	Jin Yan\footnotemark[1] \thanks{Corresponding author. Email: \href{mailto:yanj@sdu.edu.cn}{\texttt{yanj@sdu.edu.cn}}}
}
\date{}

\begin{document}
	
	\maketitle
	\begin{abstract}
		\noindent 
		The classical Corr\'adi-Hajnal theorem states that for any multiple $n$ of $3$, if $G$ is a graph with $n$ vertices and  $\delta(G) \geq 2n/3$, then $G$ can be partitioned into $n/3$ vertex-disjoint copies of the triangle [\emph{Acta Math. Acad. Sci. Hung.}, 14:423-439, 1964]. Balogh, Molla and Sharifzadeh obtained a smaller lower bound by adding the independence number condition [\emph{Random Struct. Algorithms}, 49:669-693, 2016].
	
	In this paper, we study perfect tilings in digraphs subject to conditions on the independence number and the degree. The independence number, $\alpha(D)$, of $D$ is the maximum integer $k$ such that $D$ has an independent set of cardinality $k$.
	We show that if $D$ is an $n$-vertex digraph with $\alpha(D)\leq o(1)n$ and $\delta(D) \geq (1+o(1))n$, then $D$ has a perfect $T_3$-tiling, where $T_3$ denotes a transitive triangle. This minimum degree condition is asymptotically best possible. Moreover, our result implies the theorem of Balogh, Molla, and Sharifzadeh concerning perfect triangle tilings.
	\end{abstract}
	\section{Introduction}
		Given an $n$-vertex (directed) graph $G$ and an $h$-vertex (directed) graph $H$, 
	the $minimum$ $degree$ $\delta (G)$ of $G$ is the minimum number of edges (arcs) incident to a vertex in $V(G)$ and an $H$-$tiling$ is a collection of vertex disjoint copies of $H$ in $G$. A $perfect$ $H$-$tiling$ ($H$-$factor$) is an $H$-tiling that covers all of the vertices in $V(G)$. 
	The problem of determining whether a (directed) graph $G$ contains a perfect $H$-tiling is NP-complete when $H$ is connected and $|V(H)|\ge 3$, see \cite{hell}.
	A well-known result in this topic is the Hajnal-Szemer\'edi theorem \cite{Hajnal}, which states that any $n$-vertex graph $G$ with $\delta (G) \geq (1-1/r)n$ contains a perfect $K_r$-tiling, where $r$ divides $n$. 
	This minimum degree condition is tight. 
	When $r=3$, the proof of this theorem was given by Corradi and Hajnal~\cite{corradi}. 
	Over the past period of time, there has been much work on generalizing the Hajnal-Szemer\'edi theorem, see \cite{Keevash, kier, kuhn, kuhn2, Piguet}. 
	
	A set $U$ of vertices in a (directed) graph $G$ is \emph{independent} if the (directed) graph induced by $U$ has no edges (arcs).
	The \emph{independence number}, $\alpha(G)$, of $G$ is the maximum integer $k$ such that $G$ has an independent set of cardinality $k$.
	Erd\H{o}s and S\'os \cite{Erdos} studied the maximum number of edges in an $n$-vertex $K_r$-free graph with independence number $\alpha(G)=o(1)n$. This type of problem is referred to as a Ramsey-Tur\'an problem, and has been studied extensively, see  \cite{BaloghL,ErdosHS,ErdosHSS,Mubayi,Simonovits}.
	More generally, imposing an independence number condition on an original problem gives rise to a Ramsey-Tur\'an type variant of that problem; see \cite{Balogh, BaloghFT, knier21, pi}.
	The following problem is a Ramsey-Tur\'an type variant of the Hajnal-Szemerédi theorem, proposed by Balogh, Molla and Sharifzadeh in \cite{Balogh}.
	
	\begin{question}\label{q1.1}\cite{Balogh}
		Let $G$ be an $n$-vertex graph with $\alpha(G)=o(1)n$. What is the minimum degree condition on $G$ that guarantees a perfect $K_r$-tiling in $G$ for $r \geq 3$?
	\end{question}
	
	In the same article, the authors answered Question \ref{q1.1} for $r=3$ on sufficiently large graphs.
	\begin{theorem}\cite{Balogh}\label{thm-Balogh}
		For every $\varepsilon>0$, there exist $\gamma>0$ and $n_0$ such that the following holds. For every $n$-vertex graph $G$ with $n>n_0$, if $\delta(G) \geq (1/2+\varepsilon)n$ and $\alpha(G) \leq \gamma n$, then $G$ has a perfect $K_3$-tiling.
	\end{theorem}
	
	In 2021, Knierim and Su \cite{knier21} resolved Question \ref{q1.1} for all $K_r$. Recently, Chen et al. \cite{Chen} extended the above question to any arbitrary graph $H$.
	For digraphs, the analogous problem is more complicated. Let $\delta^0(D) = \min_{v \in V(D)} \{ d^+(v),\, d^-(v) \}$ be the \emph{minimum semi-degree} of a digraph $D$. 
	Consider the digraph obtained by orienting $K_3$ either as a transitive triangle $T_3$ or as a directed $3$-cycle $C_3$.
	The tiling problems for $T_3$ and $C_3$ in digraphs with minimum (semi-)degree condition have been extensively studied; see \cite{BaloghLo, CzygrinowKM, KeevashS, wang, wangZ, Yuster}.
	It is natural to consider a Ramsey-Tur\'an type variant of the tiling problems for $T_3$ and $C_3$ in digraphs. 
	However, for $C_3$-tiling under minimum degree and independence number conditions, the counterexample in \cite{wang} shows that adding such a condition on the independence number does not improve the bound already given there.
	Hence, we consider only $T_3$-tiling in digraphs under minimum degree and independence number conditions.
	\begin{theorem}\label{thm-T3}
		For every $\varepsilon>0$, there exists $\gamma>0$ and $n_0$ such that the following holds. Let $D$ be a digraph with $n>n_0$ vertices. If $\delta(D) \geq (1+\varepsilon)n$ and $\alpha(D) \leq \gamma n$, then $D$ has a perfect $T_3$-tiling.
	\end{theorem}

Clearly, we have a corollary by Theorem \ref{thm-T3} for the semi-degree condition. 
\begin{corollary}\label{cor-T3}
	For every $\varepsilon>0$, there exists $\gamma>0$ and $n_0$ such that the following holds. For every $n$-vertex digraph $D$ with $n>n_0$, if $\delta^0(D) \geq (1/2+\varepsilon)n$ and $\alpha(D) \leq \gamma n$, then $D$ has a perfect $T_3$-tiling.
\end{corollary}
	This Corollary immediately implies the result of Balogh, Molla, and Sharifzadeh \cite{Balogh} for graphs (Theorem \ref{thm-Balogh}). Indeed, given a graph $G$ with $\delta(G) \geq (1/2+\varepsilon)n$ and $\alpha(G) \leq \gamma n$, replacing each edge by a $2$-cycle yields a digraph that satisfies the conditions of Corollary \ref{cor-T3}. Consequently, a perfect $T_3$-tiling exists, which corresponds to a perfect $K_3$-tiling of $G$.
	
	The semi-degree condition in Corollary \ref{cor-T3} is asymptotically best possible.
	So the minimum degree condition in Theorem \ref{thm-T3} is also asymptotically best possible.
	We construct a digraph $D$ to show this.
	Let $n=2k+1$ and $k \equiv 1 (\mbox{mod}\ 3)$. Then $n$ is divisible by $3$. 
	Consider a digraph $D$ formed by two complete digraphs$D_1$ and $D_2$, each of order $k+1$, with $|V(D_1)\cap V(D_2)|=1$. Let $v\in V(D_1)\cap V(D_2)$.
	It is easy to see that $\alpha(D)=2 \leq \gamma n$ and $\delta(D)=n-1$. 
	If there exists a perfect $T_3$-tiling $\mathcal{T}$ of $D$, by $k \equiv 1 (\mbox{mod}\ 3)$, there exists a copy $T\in \mathcal{T}$ which contains two vertices, one from $V(D_1)\setminus \{v\}$ and the other from $V(D_2)\setminus \{v\}$. These two vertices are not adjacent, a contradiction.
	\bigskip
	
	\textbf{Organisation.}  
	In Section 2, we introduce relevant definitions and preliminary results that will be utilized throughout the paper, followed by a brief overview of the proof framework. 
	The proof of Theorem \ref{thm-T3} is inspired by the idea of \cite{Balogh}. 
	First, the Absorbing Lemma (Lemma \ref{lem-absorb}) yields an absorbing set with the property that any small subset of the remaining vertices can be absorbed into it to form a perfect tiling. Then, applying the Almost-covering Lemma (Lemma \ref{lem-cover}) to the rest of the digraph covers almost all vertices, and the few uncovered vertices are subsequently absorbed.
	In Section 3, we present the Absorbing Lemma along with its proof. Subsequently, in Section 4, we prove the Almost-covering Lemma and employ it to establish Theorem \ref{thm-T3}. Finally, we conclude with some potential problems for future research.
	
	\section{Preparation}
	\subsection{Notation}
Most of the notations we use are standard, following the reference \cite{Diestel}.
Let $D=(V,A)$ be a digraph with vertex set $V(D)$ and arc set $A(D)$. For simplicity, we also denote a graph by $G=(V,E)$ with vertex set $V(G)$ and edge set $E(G)$, whenever the orientation of arcs is irrelevant. A \emph{semi‑complete digraph} is a digraph in which, for every pair of distinct vertices, at least one of the two possible directed arcs is present.
For each $v\in V(D)$, we use $N_D(v)$ (respectively, $N^+_D(v),\ N^-_D(v)$) to denote the \emph{neighborhood} (respectively, \emph{out-neighborhood}, \emph{in-neighborhood}) of $v$ in $D$ and $d_D(v)=|N_D(v)|$ (respectively, $d^+_D(v)=|N^+_D(v)|,\ d^-_D(v)=|N^-_D(v)|$). 
Let $N^{\Delta}_D(v)$ be the larger of $N^+_D(v)$ and $N^-_D(v)$ and  $d^{\Delta}_D(v)=\max\{d^+_D(v), d^-_D(v)\}$.
For any two disjoint sets $X, Y \subset V(D)$, let $A^+_{D}(X,Y)$ (respectively, $A^-_{D}(X,Y)$) denote the set of all arcs from $X$ to $Y$ (respectively, from $Y$ to $X$), and let $a^+_{D}(X,Y)=|A^+_D(X,Y)|$ (respectively, $a^-_{D}(X,Y)=|A^-_D(X,Y)|$).  
Let $A_{D}(X,Y)=A^+_D(X,Y)\cup A^-_D(X,Y)$ and $a_{D}(X,Y)=|A_D(X,Y)|$.
For a vertex $v\in V(D)$, let $d^+_{D}(v,X)$ (respectively, $d^-_{D}(v,X)$) denote the number of arcs from $v$ to $X$ (respectively, from $X$ to $v$).  
If there is no risk of ambiguity, we omit the subscript in the above notations.

Denote the subgraph of $D$ induced by $X$ by $D[X]$.
In this paper, if the subdigraph induced by a vertex $v$ and an arc $e$ in $D$ is a $T_3$, we simply say that $v$ and $e$ form a $T_3$ in $D$.
A \emph{path} from vertex $x$ to vertex $y$ is a sequence of distinct vertices $v_0 = x,\; v_1,\; v_2,\; \dots,\; v_n = y$ such that for each $i = 0,1,\dots,n-1$, either $(v_i, v_{i+1}) \in A(D)$ or $(v_{i+1}, v_i) \in A(D)$. We denote such a path by an $(x,y)$-path. Note that the paths considered in this paper are undirected (directions are ignored).
We call a digraph $D$ \emph{connected} if there exists a path between any two vertices in $V(D)$.
Let $[d]=\{1,2,\dots,d\}$.
The notation $\beta \ll \gamma$ indicates the existence of an increasing function $f$ such that, whenever $\beta$ and $\gamma$ are constants with $\beta \leq f(\gamma)$, the statement holds.  
The function $f$ is not always given explicitly, but it can be determined.
The constants used throughout this paper adhere to the relations below. 
\begin{equation*}
		0<\gamma \ll \zeta \ll \beta \ll \eta \ll \sigma \ll \phi \ll \psi \ll  \varepsilon <1/8.
\end{equation*}
Notably, while the specific constants represented by each symbol may vary from lemma to lemma or theorem to theorem, their order-of-magnitude relations remain invariant and obey the equations above.
	
	Now, we present some definitions that are used in the Absorbing Lemma.
	These definitions were initially given in \cite{Balogh} and we have made certain modifications, specifically changing the triangle $K_3$ to $T_3$, to make them applicable to digraphs.
	We first introduce a key definition underlying the absorption structure.
	\begin{definition}[\textbf{$(c,k)$-linked vertices}]\label{def-linked vertices}
		Let $D$ be an $n$-vertex digraph and let $x,y\in V(D)$ be two distinct vertices.
		Define $\mathcal{U}$ to be the collection of subsets $U\subseteq V(D)$ of size $3k-1$ for which both $D[U\cup\{x\}]$ and $D[U\cup\{y\}]$ contain a perfect $T_3$-tiling.
		We say that $x$ and $y$ are $(c,k)$-linked if $|\mathcal{U}|\geq (cn)^{3k-1}$.
		Each set $U\in\mathcal{U}$ is called a $k$-linking set for $\{x,y\}$ (see Figure \ref{fig-U.}).
		A set $V'\subseteq V(D)$ is $(c,k)$-linked if every pair of distinct vertices in $V'$ is $(c,k)$-linked.
		For each vertex $v\in V(D)$, let $L_{c,k}(v)$ denote the set of vertices that are $(c,k)$-linked with $v$.
	\end{definition}
\begin{figure}
	\centering
	\includegraphics[scale=0.9]{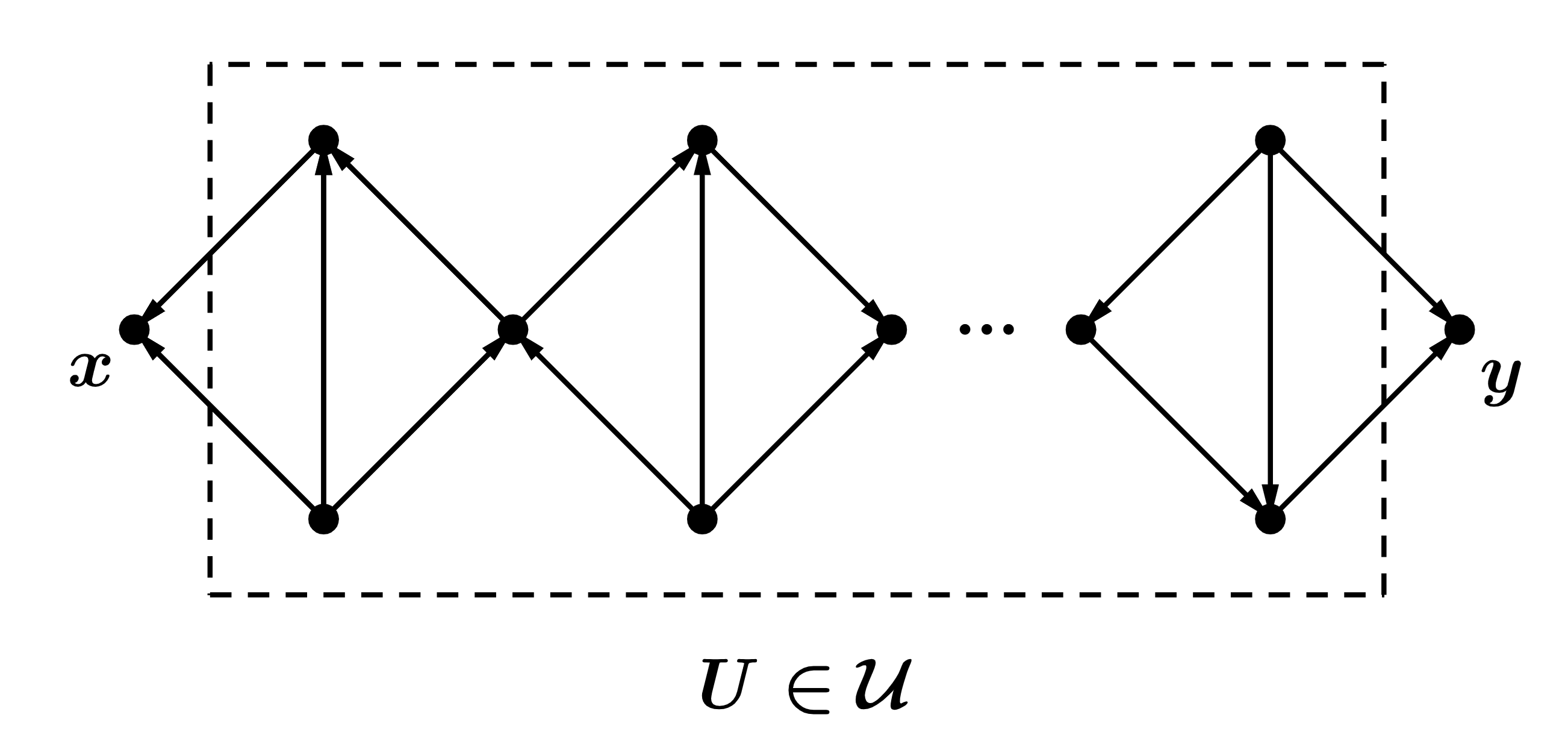}
	\caption{A $k$-linking set $U$ for $\{x,y\}$.}\label{fig-U.}
\end{figure}
	Building on the above definition, we now give the following one.
	\begin{definition}[\textbf{$(\psi,\phi,k)$-linked partition}]
		Given constants $0<\phi<\psi\leq 1$ and a positive integer $k$, let $d$ be a positive integer with $d\leq 1/\psi$.
		Let $\mathcal{P}=\{V_1,\dots,V_d\}$ be a partition of $V(D)$.
		We call $\mathcal{P}$ a $(\psi,\phi,k)$-$linked$ $partition$ if $|V_i|\geq \psi n$ and $V_i$ is $(\phi,k)$-linked for every $i\in[d]$.
	\end{definition}
Given a partition $\mathcal{P}=\{V_1,\dots,V_d\}$ of $V(D)$, let $I=(i_1,i_2,i_3)$ with $i_1,i_2,i_3\in[d]$. Such triples are used in constructing an absorbable collection for $T_3$. The more transitive triangles $T_3$ a triple $I$ yields, the more useful it is for our construction. 
Denote by $t(\mathcal{P},I)$ the number of distinct transitive triangles $T_3=v_1v_2v_3$ such that $v_j\in V_{i_j}$ for each $j\in[3]$. Since the indices in $I$ need not be distinct, for each $i\in[d]$ we let $\nu_I(i)$ be the multiplicity of $i$ in $I$.
We only consider triples $I$ with $t(\mathcal{P},I)>\phi n^3$. Accordingly, define
\begin{equation}\label{equ-Imax}
	\mathcal{I}_\phi = \bigl\{ I : t(\mathcal{P},I) > \phi n^3 \bigr\}.
\end{equation}
For every $1\leq i\leq d$, let $t(\phi,i)$ be the number of times the index $i$ appears in $\mathcal{I}_\phi$. To make the structure of these triples more transparent and easier to work with, we reorganize them into the following collection.
	\begin{definition}[\textbf{$(\mathcal{P},\phi,\eta)$-absorber}]\label{def-absorber}
		Given constants $0<\eta<\phi<\psi\leq 1$, let $\mathcal{P}=\{V_1,\dots,V_d\}$ be a partition of $V(D)$ and let $X\subseteq V(D)$.
		A collection $\mathcal{Z}$ of vertex-disjoint subsets of $V(D)\setminus X$ is called $(\mathcal{P},\phi,\eta)$-absorbable with respect to $X$ if the following holds: for each ordered triple $I=(i_1,i_2,i_3)\in\mathcal{I}_\phi$ and each $j\in[3]$, there exists a set $Z_j^I\in\mathcal{Z}$ such that $Z_j^I\subseteq V_{i_j}$ and $|Z_1^I|=|Z_2^I|=|Z_3^I|\leq \eta n$.
		Furthermore, $X$ is called an $(\mathcal{P},\phi,\eta)$-absorber if, for every collection $\mathcal{Z}$ that is $(\mathcal{P},\phi,\eta)$-absorbable (with respect to $X$), $D[X\cup V(\mathcal{Z})]$ has a perfect $T_3$-tiling.
	\end{definition}
Next, we construct an auxiliary digraph of $D$ in the following manner to illustrate the relationships among the sets in $\mathcal{Z}$.
\begin{definition}\label{def-CQ}
	Let $D$ be an $n$-vertex digraph and $\mathcal{Q}=\{Q_{1}, \dots, Q_{m}\}$ be a family of subsets of $V(D)$. Given constant $q>0$, let $C(\mathcal{Q}, q)$ be the digraph with vertex set $\mathcal{Q}$ such that:
	\begin{equation*}
		\begin{array}{rl}{(Q_{i},Q_{j})\in A(C(\mathcal {Q},q))\Leftrightarrow }&{|\{ v\in Q_{i}:|N^+(v) \cap Q_{j}| \geq qn \} | \geq qn \text{ and }}\\
			&{|\{u \in Q_{j}:|N^-(u) \cap Q_{i}| \geq qn\} | \geq qn. }\end{array}
	\end{equation*}
\end{definition}

	\subsection{Useful lemmas}
	In this subsection, we present some results on the absorbing method. 
	The following lemma is the Chernoff bound, see e.g. Theorem 2.10 in \cite{Janson}.
	
	\begin{lemma}\cite{Janson}\label{thm-Chernoff bound}
		Let $X$ be a hypergeometric random variable or let $X=\sum_{i=1}^{n} X_{i}$ where $X_{1}, \dots, X_{n}$ are independent random indicator variables. If $0<\lambda \leq 3 / 2$, then 
		\begin{equation*}
			\mathbb{P}(|X-\mathbb{E} X| \geq \lambda \mathbb{E} X) \leq 2 \exp \left(-\frac{\lambda^{2}}{3} \mathbb{E} X\right).
		\end{equation*}
		In particular, the above inequation applies when $X$ is a binomial random variable.
	\end{lemma}

	Following the proof framework in \cite{Balogh} and Definition \ref{def-linked vertices}, we obtain the following lemma. This lemma can be derived readily. Therefore, for completeness,  we provide a simple proof.

	\begin{lemma}\label{pro-Lconnect}
		Given $k_{1}$, $k_{2} \in \mathbb{N}$ and constants $c, c_1, c_{2}>0$. 
		Let $D$ be a digraph and $x_{1}, x_{2} \in V(D)$. Denote $k=k_{1}+k_{2}$ and $c'=\min \{c, c_{1}, c_{2}\}$. If 
		\[
		\left|L_{c_{1}, k_{1}}\left(x_{1}\right) \cap L_{c_{2}, k_{2}}\left(x_{2}\right)\right| \geq c n,
		\]
		then $x_{1}$ and $x_{2}$ are $\left(\frac{1}{3} c', k\right)$-linked.
	\end{lemma}
\begin{proof}
	Assume that $k_{1} \leq k_{2}$, and let $(x, U_{1}, U_{2})$ be an ordered triple such that $x \in L_{c_{1}, k_{1}}(x_{1}) \cap L_{c_{2}, k_{2}}(x_{2})$ and $U_{i}$ is a $k_{i}$-linking set for $\{x_{i}, x\}$, $i \in [2]$. Since $\left|L_{c_{1}, k_{1}}\left(x_{1}\right) \cap L_{c_{2}, k_{2}}\left(x_{2}\right)\right| \geq c n$, the number of such ordered triples is at least
	\[
	cn\cdot (c_{1}n)^{3k_{1}-1}\cdot (c_{2}n)^{3k_{2}-1}\geq (c^{\prime }n)^{3k_{1}+3k_{2}-1}.
	\]
	Then $U=\{x\} \cup U_{1} \cup U_{2}$ is a $(k_{1}+k_{2})$-linking set for $\{x_{1}, x_{2}\}$. 
	Next, we count how many $U$ are repeated.
	Suppose another such triple $(x', U_{1}', U_{2}')$ yields the same set $U$. To bound the number of repetitions, we note that $b+1 \leq 3 \cdot(3 / 2)^{b}$ for any integer $b>0$. Picking $x'$ first and then $U_{1}'$ from the remaining vertices, we obtain at most
	\[
	\left(3 k_{1}+3 k_{2}-1\right) \cdot\binom{3 k_{1}+3 k_{2}-2}{3 k_{1}-1} \leq\left(3 \cdot\left(\frac{3}{2}\right)^{3 k_{1}+3 k_{2}-2}\right) \cdot 2^{3 k_{1}+3 k_{2}-2}=3^{3 k_{1}+3 k_{2}-1}
	\]
	such triples $(x', U_{1}', U_{2}')$.
	So the number of different $U$ is at least $(c^{\prime }n/3)^{3k_{1}+3k_{2}-1}$, i.e., $x_{1}$ and $x_{2}$ are $\left(\frac{1}{3} c', k\right)$-linked.
\end{proof}

We construct an auxiliary graph for the digraph $D$ that reflects the linked relationships among the vertices.

\begin{lemma}\cite{Balogh}\label{lem-bi-M}
	Let $F$ be a bipartite graph with classes $(A, B)$ and $0<\varepsilon \leq1$ be such that $d_{F}(a) \geq\varepsilon|B|$ for every $a \in A$ and $d_{F}(b) \geq\varepsilon|A|$ for every $b \in B$. If $B$ is sufficiently large, then for every $0<\psi<\varepsilon^{4} / 64$ there exists a positive integer $d$ and a collection of disjoint subsets $\{S_{1}, \dots, S_{d}\}$ of $B$ such that
	\begin{itemize}
		\item for every $i \in[d]$, $|S_{i}| \geq\psi|B|$,
		\item $|\bigcup_{i=1}^{d} S_{i}| \geq(1-\psi)|B|$, and
		\item for every $i \in[d]$, there are at most $\psi^{3}|B|^{2}$ pairs $b, b' \in S_{i}$ such that $| N_{F}(b) \cap N_{F}(b')|<\psi^{4}|A|$.
	\end{itemize}
\end{lemma}

We also use an auxiliary digraph (as defined in Definition \ref{def-CQ}) in the following proof of the Absorbing Lemma and present one of its properties.
Note that the path which we used in the following lemma does not need to be a directed path.
\begin{figure}
	\centering
	\includegraphics[scale=0.7]{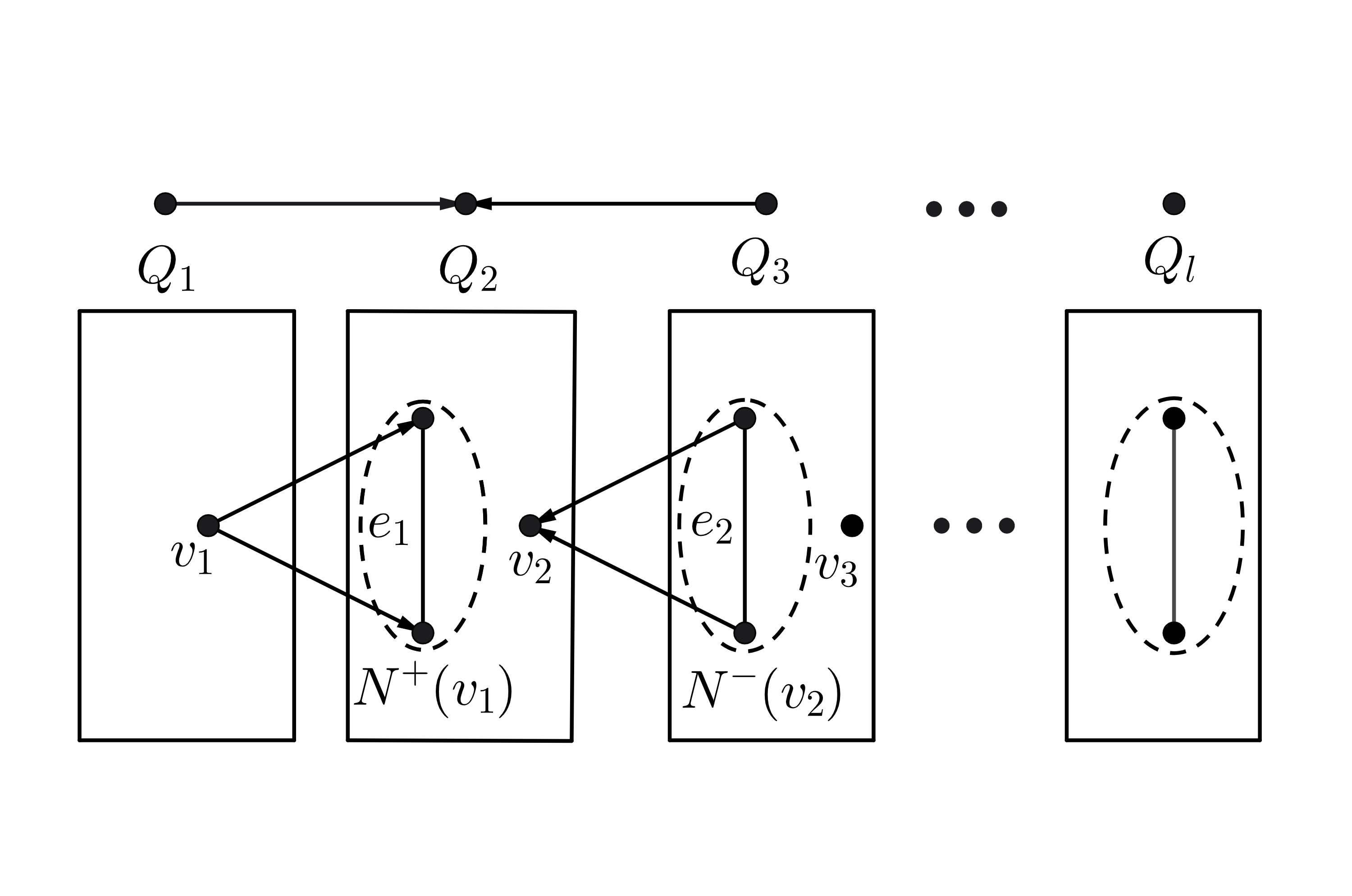}
	\caption{Relation between a path $P$ in $C(\mathcal{Q}, q)$ and disjoint $T_3$'s in $D$.}\label{fig-path.}
\end{figure}
\begin{lemma}\label{lem-path to T3}
	Given constants $0<\gamma<q<1$ and $d \in \mathbb{N}$, let $D$ be an $n$-vertex digraph such that $\alpha(D) \leq\gamma n$ and $\mathcal{Q}=\{Q_{1}, \dots, Q_{d}\}$ be a collection of pairwise disjoint subsets of $V(D)$. For any given $W \subseteq V(D)$ such that $|W|<(q-\gamma) n$, if we have a $(Q, Q')$-path $P$ in $C(\mathcal{Q}, q)$, then there exists a set $\mathcal{T}$ of vertex disjoint $T_3$'s in $G[V(\mathcal{Q}) \setminus W]$ 
	such that:
	\begin{itemize}
		\item $| \mathcal{T}| =|A(P)|$, $|V(\mathcal{T})\cap Q|=1$, $|V(\mathcal{T})\cap Q'|=2$, and
		\item $|V(\mathcal{T}) \cap Q''| \in\{0,3\}$ for every $Q'' \in \mathcal{Q}\setminus (Q\cup  Q')$.
	\end{itemize}
\end{lemma}
\begin{proof}
	Let $P$ be such a path, denote the vertices in $P$ by $Q=Q_{1}, Q_{2}, \dots, Q_{\ell}=Q'$.
	We write an ordered pair $(v,e)$ consisting of a vertex $v$ and an arc $e$ to represent the transitive triangle $T_3$ that they form.
	We construct vertex disjoint $T_3$'s $(v_{1}, e_{1}), \dots, (v_{\ell-1}, e_{\ell-1})$ (see Figure \ref{fig-path.}), where $e_{i} \in A(Q_{i+1} \setminus W )$ and $v_{i} \in Q_{i} \setminus (W \cup  V(e_{i}))$, iteratively.
	
	For $i=1$, either $Q_{1}Q_{2}$ or $Q_{2}Q_{1}$ is an arc in $C(\mathcal{Q},q)$. Without loss of generality, we assume $(Q_{1},Q_{2})$ is the arc; the other case follows by symmetry. This implies that there are at least $q n$ vertices in $Q_{1}$ each having at least $q n$ out-neighbors in $Q_{2}$. 
	Since $|W| < (q-\gamma)n$, there exists a vertex $v_{1} \in Q_{1} \setminus W$ such that $|N^+(v_{1}) \cap Q_{2}| \geq q n$, i.e., $|(N^+(v_{1}) \cap Q_{2})\setminus W | > \gamma n$.  
	By the condition $\alpha(D) \leq \gamma n$, there exists an arc $e_1 \in A(D[N^+(v_{1})])\cap Q_{2}$ such that $e_{1}$ disjoint from $W$. Clearly, $e_1$ and $v_1$ forms a $T_3$. 
	For $i\ge 2$, similarly choose $v_{i} \in Q_{i} \setminus (W \cup V(e_{i-1}))$ and $e_{i} \in A(Q_{i+1} \setminus W)$. Since $Q_{1}, Q_{2}, \dots, Q_{\ell}$ are pairwise disjoint from each other, they do not affect the degree condition. The final set of $T_3$'s $\mathcal{T} = \{v_i e_i : 1 \leq i \leq \ell-1\}$ is as desired.
\end{proof}

\subsection{Sketch proof of Theorem \ref{thm-T3}}
The proof of Theorem \ref{thm-T3} employs an absorption method, following the idea of \cite{Balogh}, and consists of two main steps. In contrast to the conventional approach, our absorbing set $U$ comprises two disjoint vertex sets: an absorber $X$ and an absorbable collection $\mathcal{Z}$.
First, we apply the Absorbing Lemma (Lemma \ref{lem-absorb}) to construct a special absorbing set $U=X\cup V(\mathcal{Z})$. This set has the property that any small subset $W$ of $V(D)\setminus U$ can be “absorbed” by $U$ to form a perfect $T_3$-tiling of $D[W\cup U]$. 
Second, we apply the Almost-covering Lemma (Lemma \ref{lem-cover}) to the remaining part of the digraph (after removing $U$). This lemma allows us to cover almost all vertices with disjoint copies of $T_3$, leaving only a small set of uncovered vertices.
Finally, since the uncovered set is small, we apply the absorbing property of $U$ to incorporate these remaining vertices. Hence, a perfect $T_3$-tiling of the entire digraph is obtained, thereby completing the proof.
\section{Absorbing Lemma}
	The constants used throughout this section satisfy the following relations. 
	\begin{equation*}
		0<\gamma \ll \zeta \ll \beta \ll \eta \ll \sigma \ll \phi \ll \psi \ll \varepsilon <1/8.
	\end{equation*}
Inspired by the Absorbing Lemma in \cite{Balogh}, we present the following digraph version, obtained by considering $T_3$-tilings under conditions of total degree $\delta(D)$ and independence number $\alpha(D)$.
\begin{lemma}[Absorbing Lemma]\label{lem-absorb}
	Given constants $0<\gamma \ll \zeta \ll \sigma \ll \varepsilon<1/8$, let $D$ be an $n$-vertex digraph with $\delta(D) \geq (1+\varepsilon)n$ and $\alpha(D) \leq \gamma n$.  
	There exists $U \subseteq V(D)$ such that $|U| \leq \sigma n$ and, for each $W \subseteq V(D)\setminus U$ that satisfies $|W|\le \zeta n$ and $|W|$ is divisible by $3$, $D[U \cup W]$ has a perfect $T_3$-tiling.
\end{lemma}

We prove Lemma \ref{lem-absorb} by the following steps.
First, by the degree condition, we construct a linked partition $\mathcal{P}$ of $V(D)$ (Lemma~\ref{lem-Pexist}). Using this partition $\mathcal{P}$, we obtain an absorber $X$ (Lemma~\ref{lem-P to X}) and then show the existence of an absorbable collection $\mathcal{Z}$ with respect to $X$ (Lemma~\ref{lem-absorbable collection Z}). Finally, for any small set $W$ disjoint from $X$ and $\mathcal{Z}$, we use $\mathcal{Z}$ to absorb $W$, thereby forming a partial $T_3$-tiling. The remaining unused vertices in $\mathcal{Z}$ form a new absorbable collection together with some copies of $T_3$. By the definition of an absorber, this yields a perfect $T_3$-tiling of $D[W\cup X\cup V(\mathcal{Z})]$.

\begin{lemma}\label{lem-Pexist}
	Given constants $0<  \phi \ll \psi \ll \mu < \varepsilon <1/8$, let $D$ be a digraph with sufficiently large order $n$. If $\delta(D) \geq(1+\varepsilon) n$, then there exists a $(\psi, \phi, 6)$-linked partition $\mathcal{P}=\{V_{1}, \dots, V_{d}\}$ of $V(D)$ for some $d \leq \varepsilon^2 /(4 \psi)$.
\end{lemma}
\begin{proof}
	Let $D$ be a digraph satisfies the condition of the lemma.\smallskip\\
	\textbf{Step 1.} The construction of an auxiliary bipartite graph $F$.
	
	The bipartite graph $F$ with parts $A(D)$ and $V(D)$ such that $e v \in E(F)$ if an arc $e\in A(D)$ and a vertex $v\in V(D)$ form a $T_3$ in $D$. Note that the auxiliary graph $F$ here is undirected and it reflect the $(\varepsilon^{2}/2, 1)$-linked relationships between the vertices (see Figure \ref{fig-F.}). 
	
	\begin{figure}
		\centering
		\includegraphics[scale=0.7]{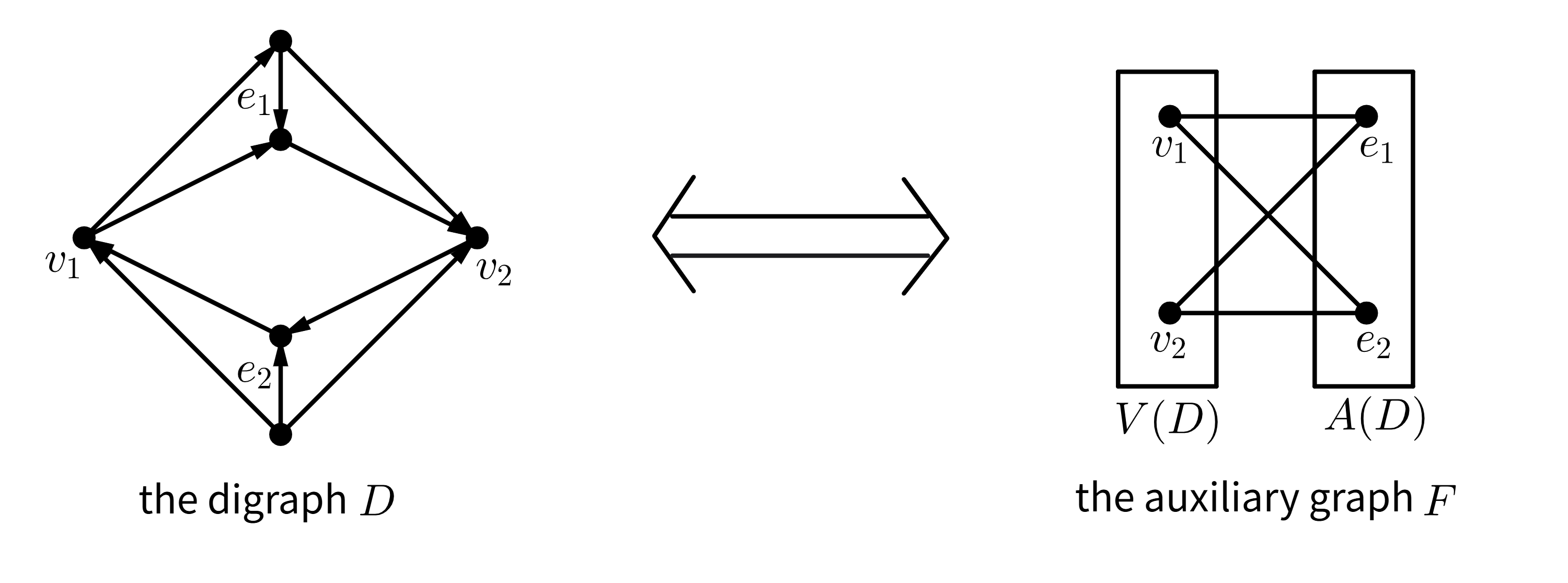}
		\caption{The auxiliary graph $F$ illustrating the $(\varepsilon^{2}/2, 1)$-linked relationships among vertices.}\label{fig-F.}
	\end{figure}

	Next, we calculate the degrees (in $F$) of the vertices in $V(D)$ and the arcs in $A(D)$.
	Recall that $N^{\Delta}(v)$ denotes the larger of the out-neighborhood and in-neighborhood of a vertex $v$, and $d^{\Delta}(v)=|N^{\Delta}(v)|$. By $\delta(D) \geq(1 +\varepsilon) n$, we have $d^{\Delta}(v)\ge\frac{1}{2}\delta(D)= \frac{1}{2}(1 +\varepsilon) n$ for every vertex $v$. For any $u$ in $N^{\Delta}(v)$, there is $|N(u)|\ge\frac{1}{2}\delta(D)= \frac{1}{2}(1 +\varepsilon) n$. 
	Then $|N^{\Delta}(v) \cap N(u)|\ge |N^{\Delta}(v)|+|N(u)|-n\ge \varepsilon n$. Summing over all $u$ in $N^{\Delta}(v)$ and noting that each arc in the subgraph induced by $N^{\Delta}(v)$ is counted at most twice, we obtain that the number of arcs in $D[N^{\Delta}(v)]$ satisfies
	$a(D[N^{\Delta}(v)]) \ge |N^{\Delta}(v)|\cdot \frac{1}{2}|N^{\Delta}(v)\cap N(u)|=\frac{1}{4}(1 +\varepsilon) \varepsilon n^2.$
	Every such arc, together with $v$, forms a transitive triangle $T_3$; consequently, $v$ is adjacent in the auxiliary graph $F$ to all these arcs, yielding $d_{F}(v) \geq\frac{1}{4}(1 +\varepsilon) \varepsilon n^2\ge \frac{1}{2}\varepsilon|A(D)|$.
	
	By the construction of $F$, for each arc $e = v_1 v_2 \in A(D)$, every vertex in $N^{+}(v_1) \cap N(v_2)$ or in $N(v_1) \cap N^{-}(v_2)$ is adjacent to $e$ in $F$. Note that $|N(v)|\ge |N^-(v)| (|N^+(v)|)$ and $|N^+(v)|+|N^-(v)|= d(v)\ge \delta(D) \geq(1 +\varepsilon) n$ for every vertex $v \in V(D)$. Then we have
	\begin{equation*}
		\begin{aligned}
			d_{F}(e)\ge &\frac{|N^+(v_1)\cap N(v_2)|+|N(v_1)\cap N^-(v_2)|}{2}-1\\
			\ge &\frac{|N^+(v_1)|+| N(v_2)|+|N(v_1)|+| N^-(v_2)|-2n}{2}-1\ge \varepsilon n-1.
		\end{aligned}
	\end{equation*}
	Choose $\mu=\varepsilon/2$ such that $0<2\psi<\mu^{4} / 64$.
	So $d_{F}(v) \geq\mu|A(D)|$ for each $v \in V(D)$ and $d_{F}(e) \geq \mu|V(D)|$ for each $e \in A(D)$.\smallskip\\
	\textbf{Step 2.} The initial partition of $V(D)$.
	
	By replacing $\varepsilon$ with $\mu$ and $\psi$ with $2\psi$ in Lemma \ref{lem-bi-M}, we can obtain a pairwise disjoint collection of vertex sets $\{V_{1}', \dots, V_{d}'\}$ of $V(D)$ that satisfies the properties stated in this lemma.
	Let $R_0=V(D) \setminus (\bigcup_{i=1}^{d} V_{i}')$. Then $|R_0| \leq2 \psi n$.  For all $i \in[d]$, we also have $|V_{i}'| \geq2 \psi n$ and all but at most $(2 \psi)^{3} n^{2}$ pairs $v, v' \in V_{i}'$ satisfy
	\begin{equation}\label{eq-NF}
		|N_{F}(v) \cap N_{F}(v')| \geq(2 \psi)^{4} n^{2}.
	\end{equation}
If $v$ and $v'$ satisfy condition (\ref{eq-NF}), then $v'\in L_{(2 \psi)^{2}, 1}(v)$. By the definition of $L_{(2 \psi)^{2}, 1}(v)$, each such vertex can be used to absorb $v$.
For each $i\in[d]$, we delete all vertices $v\in V'_i$ satisfying $|L_{(2 \psi)^{2}, 1}(v)\cap V'_i| < 8\psi^2 n$ and denote the resulting vertex set by $V_i^*$. Furthermore, let $R_i=V'_i\setminus V_i^*$.
For any pair of $v, v' \in V^*_{i}$, by $v, v' \notin R_{i}$,
	\begin{equation*}
		|L_{(2 \psi)^{2}, 1}(v) \cap L_{(2 \psi)^{2}, 1}(v')\cap V'_i| \geq 2 \cdot(|V_{i}'|-8 \psi^{2} n)-|V_{i}'| \geq 27 \phi n.
	\end{equation*}
	By Lemma \ref{pro-Lconnect} with $k_{1}=1$, $k_{2}=1$, $c=27 \phi$ and $c_{1}=c_{2}=4 \psi^{2}$, $v$ and $v'$ are $(9 \phi, 2)$-linked. Therefore, $V^*_{i}$ is $(9 \phi, 2)$-linked.\smallskip\\
	\textbf{Step 3.} Handling the remaining vertices.
	
	Before handling the remaining vertices in the set $R = R_0 \cup R_1 \cup \cdots \cup R_d$, we first prove that for every vertex $v \in V(D)$, the set $L_{\varepsilon^{2}/2, 1}(v)$, i.e., the collection of vertices that are $(\varepsilon^2/2,1)$-linked to $v$ has size at least $\frac{1}{3}\varepsilon^2 n$. This allows us to choose an index $i\in [d]$ so that $v$ can be placed into $V^*_{i}$ without affecting its linked property.
	Define
	$$J(v)=\{(w, e): w\in(V(D)\setminus v),\ e\in A(D),\ ve \text{ and } we \text{ form $T_3$}\}.$$
	By the degree of $v$ in $F$, we have 
		$$|J(v)| \geq \sum_{e\in N_{F}(v)}d_F(e)\ge \frac{1}{4}((1 +\varepsilon) \varepsilon n^2) \cdot( \varepsilon n-1) \geq \frac{1}{4}\varepsilon^{2} n^{3}.$$
	On the other hand,
	\begin{equation*}
		|J(v)| \leq(n-|L_{\varepsilon^{2}/2, 1}(v)|) \cdot(\frac{ \varepsilon^{2} n}{2})^{2}+|L_{\varepsilon^{2}/2, 1}(v)||A(D)| \leq \frac{1}{4}\varepsilon^{4} n^{3}+|L_{\varepsilon^{2}/2, 1}(v)| \frac{n^{2}}{2}.
	\end{equation*}
	Recalling that $0<\varepsilon<1 / 8$, then we have $|L_{\varepsilon^{2}/2, 1}(v)|\ge \frac{1}{3}\varepsilon^{2} n$. 
	
	Now, we deal with the vertices in $R$.
	By Lemma \ref{lem-bi-M}, at most $(2 \psi)^{3} n^{2}$ pairs $v, v' \in V_{i}'$ do not satisfy (\ref{eq-NF}). So $|R_{i}| \leq \frac{8 \psi^{3} n^{2}}{8 \psi^{2} n}=\psi n$. It also means that $|V^*_{i}| \geq \psi n \text{ for every } i \in[d].$
	
	For each $v \in R$, the number of vertices that are $(\varepsilon^2/2, 1)$-linked to $v$ is at least $\frac{1}{3}\varepsilon^2 n$. Therefore, there exists an index $i \in [d]$ such that 
	\begin{equation*}
		|\{u \in V^*_{i}: u \text{ and } v \text{ are } (\varepsilon^{2}/2, 1) \text{-linked} \}| \geq \frac{\frac{1}{3} \varepsilon^{2} n-|R_0|}{d}-|R_{i}| \geq 9 \phi n.
	\end{equation*}
	Then we put $v$ into a set $W_i$ where $i\in [d]$. From the arbitrariness of $v$, we can construct a partition (that may contain empty sets) $\{W_{1}, \dots, W_{d}\}$ of $R$ such that for each $i \in[d]$ and each $w \in W_{i}$, there is $|L_{\varepsilon^{2}/2, 1}(w) \cap V^*_{i}| \geq9 \phi n$.
	Recalling that $V^*_{i}$ is $(9 \phi, 2)$-linked, by Lemma \ref{pro-Lconnect}, for each $w \in W_{i}$, $V^*_{i} \cup \{w\}$ is $(3 \phi, 3)$-linked (also $(\phi, 6)$-linked). 
	For each pair of $w, w' \in W_{i}$, by $|V^*_{i}| \geq3 \phi n$, $w$ and $w'$ are $(\phi, 6)$-linked. 
	Let $V_{i}=V^*_{i} \cup W_{i}$ for every $i \in[d]$ and $\mathcal{P}=\{V_{1}, \dots, V_{d}\}$. 
	Since $|V^*_{i}| \geq \psi n $, we have $|V_{i}| \geq \psi n \text{ for every } i \in[d]$.
	$\mathcal{P}$ is a $(\psi, \phi, 6)$-linked partition of $V(D)$.
\end{proof}
The following Lemma~\ref{lem-P to X} shows how to obtain an absorber $X$ from the partition $\mathcal{P}$. It remains valid for digraphs because its proof depends only on the definitions of a $(\psi, \phi, k)$-linked partition and a $(\mathcal{P}, \phi, \eta)$-absorber. These definitions differ from those in \cite{Balogh} only in that $K_3$ is replaced by $T_3$; hence the lemma's proof carries over without change.
\begin{lemma}\cite{Balogh}\label{lem-P to X}
	For every $k\in \mathbb{N}$ and $0<\eta \ll \sigma \ll \phi \ll \psi \leq 1$, the following holds. Let $D$ be a digraph. If there exists a $(\psi, \phi, k)$-linked partition $\mathcal{P}=\{V_{1}, \dots, V_{d}\}$ of $V(D)$, then there exists a $(\mathcal{P}, \phi, \eta)$-absorber $X \subseteq V(D)$ such that $|X| \leq \sigma n$.
\end{lemma}

By Lemmas~\ref{lem-Pexist} and~\ref{lem-P to X}, we obtain a $(\mathcal{P}, \phi, \eta)$-absorber $X$ of $D$. We now prove that there exists a $(\mathcal{P}, \phi, \eta)$-absorbable collection $\mathcal{Z}$ with respect to $X$.
For each $I=(i_1,i_2,i_3) \in \mathcal{I}_{\phi}$ and $j \in[3]$, we let $f : \mathcal{I}_{\phi} \times [3] \to [d]$ be a function with $f(I,j)=i_j$.
Recall that in such a collection $\mathcal{Z}$, each set $Z^I_j$ is a subset of $V_{f(I,j)}$, corresponding to the coordinate $i_j$ of $I$.
We use $C(\mathcal{Z}, \beta)$ to denote the relationships among the elements of $\mathcal{Z}$, as defined in Definition~\ref{def-CQ}.
In the following lemma, property (i) ensures that $\mathcal{Z}$ can absorb any given vertex of $V(D)$, while properties (ii) and (iii) ensure that after $\mathcal{Z}$ absorbs some vertices, the remaining vertices in $V(\mathcal{Z})$ can be rearranged to form some vertex-disjoint copies of $T_3$ together with a new $(\mathcal{P},\phi,\eta)$-absorbable collection.
We then establish the existence of such a collection $\mathcal{Z}$ that is $(\mathcal{P},\phi,\eta)$-absorbable with respect to $X$.

\begin{lemma}\label{lem-absorbable collection Z}
	Let $n$ sufficiently large. Given $k\in \mathbb{N}$ and constants $0< \beta \ll \eta \ll \sigma \ll \phi \ll \psi \ll \varepsilon <1/8$, let $D$ be an $n$-vertex digraph with $\delta(D) \geq(1 +\varepsilon) n$ and $\mathcal{P}=\{V_{1}, \dots, V_{d}\}$ be a $(\psi, \phi, 6)$-linked partition of $V(D)$. For any $X \subseteq V(D)$ with $|X| \leq \sigma n$, there exists an $(\mathcal{P}, \phi, \eta)$-absorbable collection $\mathcal{Z}$ with respect to $X$ such that $|Z^I_j|=\lfloor\eta n\rfloor$ for any $I \in \mathcal{I}_{\phi}$ and $j \in[3]$ and satisfies:
	\begin{enumerate}
	\item[$($i$)$] for each $v \in V$, there exists a $Z^I_j\subset V_{i_j}$ such that $d^+_D(v, Z^I_j) \geq\beta n$, where $I \in \mathcal{I}_{\phi}$ and $j \in[3]$,
	\item[$($ii$)$] for each $I \in \mathcal{I}_{\phi}$, $Z^I_1 Z^I_2 Z^I_3$ forms a $T_3$ in $C(\mathcal{Z}, \beta)$, and
	\item[$($iii$)$] the digraph $C(\mathcal{Z}, \beta)$ is connected.
	\end{enumerate}
\end{lemma}
\begin{proof}
First, we construct a collection $\mathcal{Z}$, and then prove that it satisfies the above conditions. 
For each $V_i\in \mathcal{P}$, we delete the vertices in $X$ and denote the remaining vertex set by $V_{i}'$, i.e., $V_{i}'=V_{i} \setminus X$ for every $i \in[d]$. Choose a constant $\tau$ with $\sigma \ll \tau \ll \phi$. 
Recall that $t(\phi, i)$ is the number of the index $i$ appears in $\mathcal{I}_{\phi}$ for every $1 \leq i \leq d$. 
We select a set $Z_{i}$ of size $t(\phi, i) m$ from $V_{i}'$ uniformly at random for each $i \in[d]$ and then partition each $Z_{i}$ into $t(\phi, i)$ parts such that each part has a size of $\lfloor\eta n\rfloor$.
For each $I \in \mathcal{I}_{\phi}$ and $j \in[3]$, we uniquely assign $Z^I_j$ to one of the parts of $Z_{f(I,j)}$.
Then, we have a collection $\mathcal{Z}=\cup_{I\in \mathcal{I}_{\phi}, j\in[3]} Z^I_j$.
By Definition \ref{def-absorber}, any such $\mathcal{Z}$ corresponds to an $(\mathcal{P}, \phi, \eta)$-absorbable collection.

Since each vertex set $Z^I_j$ is randomly selected, we will show below that there exists a selection of $\mathcal{Z}$ satisfying the conditions stated in the lemma.
For any $i, j \in [d]$, define  
\[
W^+_{(i, j)}=\{v \in V_{i}': d^+_D(v, V_{j}) \geq \tau n\}.
\]  
For each $i \in [d]$, let $\mathcal{W}^+_i$ be the collection consisting of the sets $N^+_D(v) \cap V_{i}'$ for all $v \in V(D)$ and the sets $W^+_{(i,j)}$ for all $j \in [d]$. Clearly, $|\mathcal{W}^+_{i}|=n+d$, and every $W^+ \in \mathcal{W}^+_{i}$ is a subset of $V_{i}'$.
Consider the deviation event $|W^+\cap Z^I_j| < \mathbb{E}|W^+\cap Z^I_j| - \beta n.$
The random variable $|W^+ \cap Z^I_j|$ is hypergeometrically distributed for any $I \in \mathcal{I}_{\phi}$, $j \in[3]$ and $W^+ \in \mathcal{W}^+_{f(I,j)}$.
(That is, if we have a bin with $|V_{f(I,j)}|$ balls and exactly $|W^+|$ of them are red, then the probability that there are exactly $t$ red balls after drawing $m$ balls without replacement from the bin is $\Pr\bigl(|W^+\cap Z^I_j| = t\bigr)$.)
For each $W^+ \in \mathcal{W}^+_{f(I,j)}$, the expected size of the intersection is
\begin{equation}\label{eq-expected}
	\mathbb{E}|W^+\cap Z^I_j| = \frac{\lfloor\eta n\rfloor}{|V_{f(I,j)}'|}\, |W^+| \ge 0.9\eta |W^+|.
\end{equation}
Clearlay, $|W^+| \ge \beta n$. By Lemma~\ref{thm-Chernoff bound} (Chernoff bound), we have
$$\mathbb{P}(|W^+ \cap Z^I_j|<\mathbb{E}|W^+ \cap Z^I_j|-\beta n) \leq 2 \exp (-\frac{1}{3}\beta^{2} \mathbb{E}|W^+ \cap Z^I_j|) \leq 2 \exp(-\frac{1}{3}\beta^{3} n).$$
Applying the union bound over all $n+d$ sets $W^+$ in $\mathcal{W}^+_{f(I,j)}$, we conclude that with high probability, $|W^+\cap Z^I_j|\geq \mathbb{E}|W^+\cap Z^I_j|-\beta n$ holds for every such $W^+$ simultaneously.
Then, by (\ref{eq-expected}), there exists an $(\mathcal{P}, \phi, \eta)$-absorbable collection $\mathcal{Z}$ such that, for each of the at most $d^{3} \leq1 / \phi^{3}$ elements $I \in \mathcal{I}_{\phi}$ and every $j \in[3]$, $|W^+ \cap Z^I_j|\geq 0.9 \eta |W^+|-\beta n  \text{ for every } W^+ \in \mathcal{W^+}_{f(I,j)}.$
Recall that, for each $v\in V(D)$, $I \in \mathcal{I}_{\phi}$ and $j \in[3]$, $N^+_D(v) \cap V_{f(I,j)}'\subseteq \mathcal{W}^+_{f(I,j)}$. Then, we have
	\begin{equation}\label{eq-d+v XI}
		|N^+_D(v) \cap Z^I_j| \geq 0.9  \eta |N^+_D(v) \cap V_{f(I,j)}'|-\beta n.
	\end{equation}

Recall that $W^+_{(f(I,j), i')} \subseteq \mathcal{W}^+_{f(I,j)}$ for any $i' \in[d]$; hence, for every $I \in \mathcal{I}_{\phi}$ and $j \in[3]$, 
	\begin{equation*}\label{eq-W+ XI}
		|W^+_{(f(I,j),i')}\cap Z^I_j|\geq 0.9 \eta |W^+_{(f(I,j),i')}|-\beta n.
	\end{equation*}
For any pair of $I, I' \in \mathcal{I}_{\phi}$ and $j, j' \in[3]$, we have
	\begin{equation}\label{arc M-N}
		\begin{aligned}
			& \text{if } f(I,j) \neq f(I',j') \text{ and } (V_{f(I,j)}, V_{f(I',j')}) \in A(C(\mathcal{P}, \tau)) , then \\
			& (Z^I_j, Z^{I'}_{j'}) \in A(C(\mathcal{Z}, \beta)).
		\end{aligned}
	\end{equation}

(i).
We claim that $t(\phi, i) \geq1$ for all $i \in[d]$. Otherwise, suppose $j \in[d]$ such that $t(\phi, j)=0$. The number of $T_3$ containing vertices in $V_{j}$ is less than $d^{2} \phi n^{3}$. But there are at least $(\sum_{v \in V_{j}} \frac{1}{3}a(G[N^{\Delta}(v)])) \geq \frac{1}{4}\psi n \cdot \varepsilon n^2$ such $T_3$ in $V_j$, a contradiction.
Thus, for every $i \in[d]$, there exists a $I \in \mathcal{I}_{\phi}$ and $j\in [3]$ such that $Z^I_j \subseteq V_{i}$. 
Moreover, for each vertex $v \in V(D)$, there exists $i \in[d]$ satisfying $d^+_D(v, V_{i}') \geq((1 +\varepsilon) n-2|X|) / 2d \geq\phi n$.
By (\ref{eq-d+v XI}), this implies $d^+_D(v, Z^I_j) \geq\beta n$, and so (i) holds for this $\mathcal{Z}$.

(ii).
Let $I=(i_1,i_2,i_3) \in \mathcal{I}_{\phi}$.
By (\ref{arc M-N}), if $|W^+_{(f(I,j), f(I,j'))}| \geq\tau n / 2$ for any $j,j'\in [3]$ and $j<j'$, then $Z^I_1 Z^I_2 Z^I_3$ forms a $T_3$ in $C(\mathcal{Z}, \beta)$. By $I \in \mathcal{I}_{\phi}$ and (\ref{equ-Imax}), there are at least $\phi n^{2}$ arcs from $V_{f(I,j)}$ to $V_{f(I,j')}$. Then, since $d^+_{D}(v, V_{f(I,j')}) \leq|V_{f(I,j')}| \leq n$ for every $v \in W^+_{(f(I,j), f(I,j'))}$, 
\begin{equation}
	\begin{split}
		|W^+_{(f(I,j), f(I,j'))}| & \geq \frac{a^+_{D}(W^+_{(f(I,j), f(I,j'))}, V_{f(I,j')})}{n} \\
		&=\frac{a^+_{D}(V_{f(I,j)}, V_{f(I,j')})-a(V_{f(I,j)} \setminus W^+_{(f(I,j), f(I,j'))}, V_{f(I,j')})}{n}\\ 
		&\geq \frac{\phi n^{2}-\tau n \cdot(|V_{f(I,j)}|-|W^+_{(f(I,j), f(I,j'))}|)}{n}\geq \frac{\tau n}{2}.
	\end{split}
\end{equation}
The case of $\nu_{I}(i) \geq 2$ could be $f(I,j)=f(I,j')$. So we have $|W^+_{(f(I,j), f(I,j'))}| \geq\tau n / 2$ for any $j,j'\in [3]$ and $j<j'$, i.e., (ii) holds for this $\mathcal{Z}$.

(iii).
When $d=1$, by the above discussion, (\ref{arc M-N}) and the definition of $\mathcal{P}$, $C(\mathcal{Z}, \beta)$ is isomorphic to $T_{3}$, so (iii) holds.
Suppose that $d\ge 2$. Recall that, for any $i \in[d]$, there exists a $I \in \mathcal{I}_{\phi}$ and $j\in [3]$ such that $Z^I_j \subseteq V_{i}$. 
Let $I, I' \in \mathcal{I}_{\phi}$ and $f(I,j)\neq f(I',j')$.
Then, by (\ref{arc M-N}), if there exists a path from $V_{f(I,j)}$ to $V_{f(I',j')}$ in $C(\mathcal{P}, \tau)$, then we can find a path from $Z^I_j$ to $Z^{I'}_{j'}$ in $C(\mathcal{Z}, \beta)$. 
It means that, if digraph $C(\mathcal{P}, \tau)$ is connected, then $C(\mathcal{Z}, \beta)$ is connected. So we only need to show that $C(\mathcal{P}, \tau)$ is connected.

Partition $\mathcal{P}$ arbitrarily into two subfamilies $\mathcal{Q}_1$ and $\mathcal{Q}_2$, and let $Q_i = \bigcup \mathcal{Q}_i$ for $i = [2]$. To prove that $C(\mathcal{P}, \tau)$ is connected, it suffices to show that there is an arc between $\mathcal{Q}_1$ and $\mathcal{Q}_2$.
Without loss of generality, assume $|Q_1| \leq |Q_2|$; then $|Q_1| \leq n/2$.
Let $V_{1} \in \mathcal{Q}_{1}$. For every $v \in V_{1}$, since $\delta(D) \geq (1+\varepsilon)n$ and $|Q_1| \leq n/2$, we have $|N_D(v) \cap Q_2| \geq\frac{1}{2}(\delta(D)-2|Q_1|)\ge  \frac{1}{2}\varepsilon n$. 
Summing over all $v \in V_1$ gives $a^+_{D}(V_{1}, Q_{2})+a^-_{D}(V_{1}, Q_{2}) \geq\frac{1}{2}\varepsilon n|V_{1}|$.
By the pigeonhole principle, there exists a $V_2 \in \mathcal{Q}_2$ such that $\max\{a^+_{D}(V_{1}, V_{2}),a^-_{D}(V_{1}, V_{2})\} \geq\frac{1}{4d}\varepsilon n |V_{1}|$.
We may assume without loss of generality that $a^+_{D}(V_{1}, V_{2})\ge \frac{1}{4d}\varepsilon n |V_{1}|$; consequently, $a^-_{D}(V_{2}, V_{1})\ge \frac{1}{4d}\varepsilon n |V_{1}|$ as well.
Now define two quantities:
\begin{itemize}
	\item $b_1$: the number of vertices $v \in V_1$ with $|N^+_D(v)\cap V_2|\ge \tau n$;
	\item $b_2$: the number of vertices $u \in V_2$ with $|N^-_D(u)\cap V_1|\ge \tau n$.
\end{itemize}
From the definition of $b_1$ and the bound on the total arcs from $V_1$ to $V_2$, we obtain
$$b_1|V_{2}|+\tau n (|V_{1}|-b_1)  \geq\frac{1}{4d}\varepsilon n |V_{1}|,$$
and similarly,
$$b_2|V_{1}|+\tau n (|V_{2}|-b_2)  \geq\frac{1}{4d}\varepsilon n |V_{1}|.$$
Recall that $\psi n \leq|V_{1}|, |V_{2}| \leq n$ and $\psi \varepsilon / d \geq\psi^{2} \varepsilon \geq2 \tau$. These inequalities show that
\begin{equation*}
	b_1\geq \frac{\varepsilon n |V_{1}|/ 4d-\tau n |V_{1}|}{|V_{2}|-\tau n} \geq \frac{(\psi \varepsilon / 4d-\tau) n^{2}}{n} \geq \tau n,
\end{equation*}
and analogously, 
\begin{equation*}
	b_2 \geq \frac{\varepsilon n|V_{1}| / 4d-\tau n |V_{2}|  }{|V_{1}|-\tau n} \geq \frac{(\psi \varepsilon / 4d-\tau) n^{2}}{n} \geq \tau n.
\end{equation*}
Thus, there exists an arc $(V_1,V_2)\in A(C(\mathcal{P}, \tau))$. Hence, $C(\mathcal{P}, \tau)$ is connected, and consequently $C(\mathcal{Z}, \beta)$ is connected as well. So (iii) holds for this $\mathcal{Z}$.
\end{proof}
\begin{proof}[\textbf{Proof of Lemma \ref{lem-absorb}}]
	Given constants $\beta, \eta, \phi$ and $\psi$ such that $0<\gamma \ll \zeta \ll \beta \ll \eta \ll \sigma \ll \phi \ll \psi \ll \varepsilon <1/8$. Let $\sigma'=\frac{1}{2}\sigma$.
	First, we apply Lemma~\ref{lem-Pexist} to obtain a $(\psi, \phi, 6)$-linked partition $\mathcal{P}$ of $V(D)$. Using this partition, Lemma~\ref{lem-P to X} then yields a set $X\subseteq V(D)$ with $|X| \leq\sigma' n$ that is an $(\mathcal{P}, \phi, \eta)$-absorber. Finally, by Lemma~\ref{lem-absorbable collection Z}, we select a collection $\mathcal{Z}$ of disjoint subsets of $V(D)\setminus X$ that is $(\mathcal{P},\phi,\eta)$-absorbable with respect to $X$ and satisfies properties (i), (ii), and (iii) of Lemma~\ref{lem-absorbable collection Z}.
	Let $Z=V(\mathcal{Z})=\cup \{Z^I_j: I \in F_{\phi}, j \in[3]\}$ and $U=X \cup Z$.
	Note that  $|U| \leq\sigma' n+3 d^{3} \eta n \leq \sigma n$.
	
	For each $W \subseteq V(D) \setminus U$ which satisfies $|W| \leq\zeta n$ and $|W|$ is divisible by $3$, we achieve a perfect $T_3$-tiling of $G[W \cup U]$ through the following three steps.\smallskip\\
	\textbf{Step 1.} Use $\mathcal{Z}$ to absorb $W$.
	
	By Lemma \ref{lem-absorbable collection Z} (i), for any vertex $w\in W$, there exists $I \in F_{\phi}$ and $j \in[3]$, such that $d^+_D(w, Z^I_j) \geq \beta n>\gamma n+2|W|$. 
	Since $\alpha(D)\le\gamma n$, for each vertex $w\in W$ we can select an arc $e_{w} \in A(D[N^+_D(w) \cap Z^I_j])$ such that $we_{w}$ forms a $T_3$ in $D$, and moreover, the arcs $e_w$ are distinct for distinct $w\in W$. Consequently, the family $\mathcal{T}_W=\{we_{w}: w \in W, e_{w}\in A(D[N^+_D(w) \cap Z^I_j])\}$ consists of vertex-disjoint copies of $T_3$.\smallskip\\
	\textbf{Step 2.} Adjust the remaining vertices in $\mathcal{Z}$ to form a $T_3$-tiling and a new $(\mathcal{P}, \phi, \eta)$-absorbable collection with respect to $X$.
	
	In this step, we iteratively construct a set $\mathcal{T}_0$ of vertex-disjoint copies of $T_3$ in $D[Z \setminus V(\mathcal{T}_W)]$, such that the remaining vertices form a new $(\mathcal{P}, \phi, \eta)$-absorbable collection with respect to $X$.
	For convenience, we will use $\mathcal{T}_0$ to represent the $T_3$ that have been constructed so far in this iterative process. 
	At the beginning of this process, let $\mathcal{T}_0=\emptyset$.
	For each $I \in F_{\phi}$ and $j \in[3]$, define
	\[
	Z^{I*}_{j}=Z^I_j \setminus V(\mathcal{T}_W \cup \mathcal{T}_0),\qquad
	\mathcal{Z}'= \{Z^{I*}_{j} : I \in F_{\phi}, j \in[3]\},\qquad
	Z'=V(\mathcal{Z}').
	\]
	We now show that for every $I\in F_{\phi}$, the three sets $Z^{I*}_1$, $Z^{I*}_2$, and $Z^{I*}_3$ have the same cardinality. Consequently, $\mathcal{Z}'$ forms a new $(\mathcal{P}, \phi, \eta)$-absorbable collection with respect to $X$. The construction of $\mathcal{T}_0$ involves a two-stage process.

	In the first stage, we show that
	\begin{equation}\label{conc-stage1}
		|Z^{I*}_{1} \cup Z^{I*}_{2} \cup Z^{I*}_{3}| \equiv 0 \quad(\bmod\ 3) \text{ for any } I \in F_{\phi}.
	\end{equation}
	A vector $I\in F_{\phi}$ is called \emph{bad} if it does not satisfy (\ref{conc-stage1}) in any step of the first stage of the algorithm.
	Since $|Z'|=|Z|-2|W|-3|\mathcal{T}_0|$ and both $|W|$ and $|Z|$ are divisible by $3$, $|Z'|$ is always divisible by $3$.
	Then, for each bad $I \in F_{\phi}$, if $|Z^{I*}_{1} \cup Z^{I*}_{2} \cup Z^{I*}_{3}| \equiv 1\ (\bmod\ 3)$, there exists another bad vector $I' \in F_{\phi}\setminus \{I\}$.
	By Lemma \ref{lem-absorbable collection Z} (iii), there exists a path $P$ from $Z^I_1$ to $Z(I', 1)$ in the digraph $C(\mathcal{Z}, \beta)$.
	Then, by Lemma \ref{pro-Lconnect}, we can find a $T_3$-tiling corresponding to path $P$ in $D$, which contains one vertex in $|Z^{I*}_{1} \cup Z^{I*}_{2} \cup Z^{I*}_{3}|$ and two vertices in $|Z'(I', 1) \cup Z'(I', 2) \cup Z'(I', 3)|$ and add it to $\mathcal{T}_0$ does not increase the number of bad $I$'s. Then we add it to $\mathcal{T}_0$ and at least one of $I$ and $I'$ is no longer bad after this step.
	Note that, for each $I\in  F_{\phi}$ which was not bad before this step, it remains good after this step. 
	We can finish this process within at most $|Z|$ steps.
	So $|\mathcal{T}_0| \leq|Z|(|Z|-1) \leq(3 \cdot d^{3})^{2}$ after the first stage.
	
	The next goal is to achieve $|Z^{I*}_1| = |Z^{I*}_2| = |Z^{I*}_3|$ for every $I\in F_\phi$. Fix an $I\in F_\phi$ for which this equality does not hold; such an $I$ satisfies condition~(\ref{conc-stage1}). By Lemma~\ref{lem-absorbable collection Z}(ii), the triple $Z^I_1 Z^I_2 Z^I_3$ forms a $T_3$ in $C(\mathcal{Z}, \beta)$. 
	Let $j, j', j''$ be pairwise distinct indices in $[3]$ satisfying $|Z^{I*}_j| \ge |Z^{I*}_{j'}| \ge |Z^{I*}_{j''}|$. Then, by the construction of $C(\mathcal{Z}, \beta)$, there exists a vertex $v \in Z^{I*}_j$ such that either $d^+_D(v, Z^{I*}_{j'})$ or $d^-_D(v, Z^{I*}_{j'})$ exceeds $\gamma n = \alpha(D)$. Hence, we can find a copy of $T_3$ with one vertex in $Z^{I*}_j$ and two vertices in $Z^{I*}_{j'}$, and we add this $T_3$ to $\mathcal{T}_0$.
	
	We now show that this process can be repeated until $|Z^{I*}_1| = |Z^{I*}_2| = |Z^{I*}_3|$ holds for every $I\in F_\phi$.
	Let $c_I=(|Z^{I*}_{j}|-|Z^{I*}_{j''}|)+(|Z^{I*}_{j'}|-|Z^{I*}_{j''}|) $.
	Fixed $c'_I=c_I$ before any $T_3$ being picked (Note that $c'_I$ and $c_I$ are divisible by $3$ after each steps).
	Note that $|\mathcal{T}_0| \leq9 d^{6}$ at the start of the second stage of the algorithm, $|\mathcal{T}_W|=|W|$. 
	Then, since every $T_3\in \mathcal{T}_0 \cup \mathcal{T}_W$ has at most $2$ vertices in $Z^I_j$ for any $j \in[3]$, we have that $c'_I \leq 4(9  d^{6}+|W|)<2 \zeta n$. Therefore, $V(\mathcal{T}_0 \cup \mathcal{T}_W)$ intersects any set in $\mathcal{Z}$ at most $2(\frac{2}{3} c'_I +9 d^{6}+|W|)<(\beta-\gamma) n$ vertices.
	$c_I$ decreases by $3$ after each $T_3$ is added to $\mathcal{T}_0$ when $|Z^{I*}_{j}|\neq |Z^{I*}_{j'}|$. When $|Z^{I*}_{j}|=|Z^{I*}_{j'}|$, $c_I$ does not change, but $c_I$ will decrease by $3$ when the following $T_3$ is added to $\mathcal{T}_0$.
	Then this process terminates after constructing at most $\frac{2}{3} c'_I$ copies of $T_3$. Consequently, we have
	\begin{equation*}
		|V(\mathcal{T}_0 \cup \mathcal{T}_Y) \cap Z^I_j|,|V(\mathcal{T}_0 \cup \mathcal{T}_Y) \cap Z^I_{j'}|<(\beta-\gamma) n,
	\end{equation*}
	which means that, for any $j, j'\in [3]$ with $j\neq j'$, there exists $v \in Z^{I*}_{j}$ such that $d^+_D(v, Z^{I*}_{j'})$ or $d^-_D(v, Z^{I*}_{j'})>\gamma n=\alpha(D)$ after any process.
	We repeatedly select a copy of $T_3$ and add it to $\mathcal{T}_0$ until $|Z^{I*}_{1}|=|Z^{I*}_{2}|=|Z^{I*}_{3}|$.
	Now, by Lemma~\ref{lem-absorbable collection Z}, we have $|Z^{I*}_1| = |Z^{I*}_2| = |Z^{I*}_3| \le \eta n$. So the collection $\mathcal{Z}'$ is $(\mathcal{P}, \phi, \eta)$-absorbable with respect to $X$.\smallskip\\
	\textbf{Step 3.} Use $X$ to absorb $\mathcal{Z}'$.
	
	Recalling that $X$ is an $(\mathcal{P}, \phi, \eta)$-absorber, by the definition of $X$, there exists a perfect $T_3$-tiling of $D[X \cup Z']$. This $T_3$-tiling together with $\mathcal{T}_W\cup\mathcal{T}_0$ is a perfect $T_3$-tiling of $D[W \cup X \cup Z]=G[W \cup U]$, which completes the proof.
\end{proof}
\section{Almost-covering Lemma and the Proof of Theorem~\ref{thm-T3}}
In this section, we prove Almost-Covering Lemma and apply it to establish the main theorem of this paper. First of all, we give an obvious proposition, which is used to prove the Almost-covering Lemma.
\begin{proposition}\label{pro-4arcT3}
	Let $H$ be a semi-complete digraph with $V(H)=\{v_1,v_2,v_3\}$. If $|A(H)|\ge 4$, then $H$ has a copy of $T_3$ as a subgraph. 
\end{proposition}

\begin{lemma}[Almost-covering Lemma]\label{lem-cover}
	Given constant $\varepsilon>0$, there exist $\gamma>0$ and $n_0$ such that the following holds. Let $D$ be an $n$-vertex digraph with $n>n_0$. If $\delta(D) \geq (1+\varepsilon)n$ and $\alpha(D) \leq \gamma n$, then there is a perfect $T_3$-tiling cover all but at most $12/\varepsilon+1$ vertices.
\end{lemma}
\begin{proof}
Given constant $\gamma<\varepsilon / 72$, since $\delta(D) \geq (1+\varepsilon)n$, there is $|N^{\Delta}(v)|=d^{\Delta}(v) \ge \frac{1}{2}\delta(D)$ for each vertex $v\in V(D)$. Then, by $\alpha(D) \leq \gamma n$, there exists an arc $e\in D[N^{\Delta}(v)]$. Note that $v$ and $e$ form a $T_3$ in $D$, which indicates the existence of $T_3$-tiling in $D$.
	Let $\mathcal{T}$ be a maximal family of disjoint $T_3$ in $D$, and let $M$ be a maximal matching in $D[V(D) \setminus V(\mathcal{T})]$. Denote $V_0=V(D-\mathcal{T}-M)$. Set $t = |\mathcal{T}|$ and $m = |M|$. Then clearly $n=3 t+2 m+|V_0|$, and $|V_0| \leq \alpha(D) \leq \gamma n$.
	Note that, for each vertex $v\in V(D)$, if there exists an arc in $D[N^+(v)]$ or $D[N^-(v)]$, we can find a copy of $T_3$ in $D$ contain $v$. So $t\geq\frac{\delta(D)/2-\alpha(D)}{3} \geq \frac{n}{6}$.
	We only need to prove the following:
	\begin{equation*}
		\begin{aligned}
		&\text{(I)} m < 6 / \varepsilon;\\
		&\text{(II)} V_0 \le 1.
		\end{aligned}
	\end{equation*}

(I). Suppose $\varepsilon m \geq 6$.
	By the maximality of $\mathcal{T}$, every vertex $v\in V(M)$ satisfies 
	$d_D(v, V(M)\cup V_0) \le 2m+2|V_0|$. 
	Combining this with $|V_0| \leq\gamma n$, $t\ge \frac{n}{6}$, and $\gamma<\varepsilon / 72$,
	\begin{equation*}
	\begin{aligned}
		d(v, V(\mathcal{T})) & \geq(1+\varepsilon) n-2|V_0|-2m=(1+\varepsilon)(3 t+2 m+|V_0|)-2|V_0|-2m \\
		& \geq 3(1+\varepsilon)t-|V_0| \ge (3+2\varepsilon)t.
	\end{aligned}
	\end{equation*}
	Then 
	\begin{equation}\label{eq-arcMT}
	a_{D}(V(M),V(\mathcal{T}))=\sum_{v\in V(M)} d_D(v, V(\mathcal{T}))\ge (6+4\varepsilon)tm.
	\end{equation}
	This implies that there exists a $\mathcal{T}'\subset \mathcal{T}$ such that $a_{D}(T,V(M))\ge \frac{1}{t}a_{D}(V(M),V(\mathcal{T}))\ge 6m+18$ for each $T\in \mathcal{T}'$. 
	So there exists a vertex $y_T \in V(T)$ such that $a_{D}(T,V(M))\ge 2m+6$, and then $|N^{\Delta}(y_T)\cap V(M)|\ge m+3$. Let $Y=\{y_{T}: T \in \mathcal{T}'\}$, $t'=|\mathcal{T}'|$ and $R=V(\mathcal{T}') \setminus Y$.
	Since each $T \in \mathcal{T}$ can send at most $12m$ arcs to $M$, 
	\begin{equation*}
		a_{D}(V(M),V(\mathcal{T})) \leq 12 mt'  +(t-t')(6 m+17)=(6 m+17) t+(6 m-17) t'.
	\end{equation*}
By $\varepsilon m \geq 6$ and (\ref{eq-arcMT}), we have
\begin{equation}\label{eq=t'}
	t' \geq \frac{4\varepsilon m-17}{6 m-17} \cdot t \geq \frac{\varepsilon m}{6 m-17} \cdot t \geq \frac{\varepsilon t}{6}  \geq \frac{\varepsilon n}{36}.   
\end{equation}
	By $a_{D}(V(M), V(\mathcal{T}')) \geq(6 m+18) t'$, there exists $x \in V(M)$ such that 
	\begin{equation*}
		\max\{d^+_D(x, V(\mathcal{T}')), d^-_D(x, V(\mathcal{T}'))\}\ge \frac{a(V(M), V(\mathcal{T}'))}{4 m} \geq \frac{(6 m+18) t'}{4 m} \geq \frac{3t'}{2}.
	\end{equation*}
	Without loss of generality, assume $d^+_D(x, V(\mathcal{T}'))\ge 3t'/2$.
	Recall that for each $T \in \mathcal{T}'$, we have $|V(T) \cap Y| = 1$. Consequently, $d^+_D(x, R) \geq d^+(x, V(\mathcal{T}'))-|Y| \geq t'/2 \geq (\varepsilon n) / 72>\gamma n$.
	By $\alpha(D) \leq\gamma n$, there exists an arc $e_x$ in $D[N^+_D(x)\cap R]$, and $e_x$ and $x$ forms a $T_3$, denoted it as $T_x$.
	If there exists a $T_0\in \mathcal{T}'$ such that $|V(T_0)\cap V(T_x)|=2$, then by the definition of $y_T$, $y_T\notin V(T_x)$, we can find a $T_3$ containing $y_T$ and an arc in $M$, denoted it as $T_y$. Clearly, $\mathcal{T}\cup \{T_x, T_y\}\setminus T$ is a $T_3$-tiling of $D$ and it has more elements than $\mathcal{T}$, a contradiction.
	Thus $T_x$ uses vertices from two distinct elements $T_1, T_2 \in \mathcal{T}'$, and for each $i \in \{1,2\}$, the vertex $y_{T_i}$ has three arcs in $M$ that can be combined with it to form a $T_3$. Therefore, we greedily select two arcs in $M$ that respectively form $T_3$ with $y_{T_1}$ and $y_{T_2}$, and denote them as $T'_1, T'_2$. Then $\mathcal{T}\cup \{T_x, T'_1, T'_2\}\setminus \{T_1, T_2\}$ is a $T_3$-tiling of $D$, and its cardinality is greater than that of $\mathcal{T}$, a contradiction.
	It means that $m<6 / \varepsilon$.

(II). Suppose there exist two distinct vertices $v_1, v_2 \in V_0$. For each $i\in \{1,2\}$, we have $a_{D}(v_i ,V(\mathcal {M}))\le 2m$. By the maximality of $\mathcal {M}$, $V_0$ is an independent set. Then we have
\begin{equation*}
	a_{D}(\{v_1, v_2\} ,V(\mathcal{T}))\geq 2(\delta (D)-2m)\geq 2 \left( (1+\varepsilon)n - \frac{12}{\varepsilon}\right) \geq (2+\varepsilon)n>6t+\varepsilon n.
\end{equation*}
So there exists a $\mathcal{T}''\subset \mathcal{T}$ such that $a_{D}(\{v_1, v_2\},V(T))\ge 7$ for each $T\in \mathcal{T}''$.
Let $t''=|\mathcal{T}''|$. By the pigeonhole principle and $\gamma<\varepsilon / 72$, 
\begin{equation}\label{eq-t''}
	t''\ge \varepsilon n/3>\gamma n.
\end{equation}
For each $T\in \mathcal{T}''$, without loss of generality, we assume $d_D(v_1,V(T))\ge d_D(v_2,V(T))$.
		
If $d_D(v_1,V(T))\ge 5$, then, by $d_D(v_1,V(T))\le 6$, it follows that $d_D(v_2,V(T))\ge 1$. Let $V(T)=\{x,y,z\}$ with the property that $z$ is adjacent to $v_2$. Consequently, $d_D(v_1,\{x,y\})\ge d_D(v_1,T)-2\ge 3$, and by Proposition~\ref{pro-4arcT3}, $xy v_1$ forms a $T_3$. Then we have a $T_3$ and an arc in $D[\{v_1, v_2\}\cup V(T)]$, a contradiction.
So $d_D(v_1,V(T))= 4$ and $d_D(v_2,V(T))\ge 3$. Without loss of generality, we assume $d_D(v_1,x)= 2$. By $d_D(v_2,V(T))\ge 3$, there exist a vertex in $V(T)\setminus\{x\}$ adjacent to $v_2$.
Suppose $z$ is adjacent to $v_2$. 
If $y$ is adjacent to $v_1$, then $xyv_1$ forms a $T_3$, and together with an arc in $D[\{v_1, v_2\}\cup V(T)]$ this yields a contradiction. Hence $y$ is not adjacent to $v_1$, and consequently $d_D(v_1,z)=d_D(v_1,V(T))-d_D(v_1,\{x,y\})= 2$.
If $y$ is adjacent to $v_2$, then $xzv_1$ forms a $T_3$ while $yv_2$ is an arc, again a contradiction. Therefore, by Proposition~\ref{pro-4arcT3}, we have $d_D(y, \{v_1, v_2\}) = 0$, and both $xzv_1$ and $xzv_2$ form a $T_3$.
(The case where $y$ is adjacent to $v_2$ is symmetric.)

Then, for each $T\in \mathcal{T}''$, denote $V(T)=\{x,y,z\}$, we can get the following conclusion: There exist a vertex $y\in V(T)$ such that $d_D(y,\{v_1, v_2\})=0$; both $xzv_1$ and $xzv_2$ form a $T_3$, where $\{x,z\}=V(T)\setminus \{y\}$.
let $Y'$ be the set of such vertex $y\in V(T)$ for any $T \in \mathcal{T}''$. By (\ref{eq-t''}), $|Y'|=t'' > \gamma n$.
By $\alpha(D) \leq \gamma n$, there exists an arc $yy' \in A(Y')$. Denote $T$ (respectively, $T'$) be the $T_3$ in $\mathcal{T}''$ containing $y$ (respectively, $y'$), and $\{x,z\}=V(T)\setminus \{y\}$  (respectively, $\{x',z'\}=V(T')\setminus \{y'\}$). Then $xzv_1$ and $x'z'v_2$ form two $T_3$ and $yy' \in A(D)$, contradicting the maximality of $\mathcal{T}$. 
So $|V_0|\le 1$.
	
By (I) and (II), the number of vertices not covered by $\mathcal{T}$ is less than $2m + |V_0|< 2 \cdot 6/\varepsilon + 1 = 12/\varepsilon + 1$, which completes the proof. 
\end{proof}
We now use Lemmas \ref{lem-absorb} and \ref{lem-cover} to prove Theorem \ref{thm-T3}. 
\begin{proof}[\textbf{Proof of Theorem \ref{thm-T3}}]
	Given constants 
	$$0<\gamma \ll \zeta \ll \sigma \ll \varepsilon<1/6$$
	be as in Lemma \ref{lem-absorb}, among them, $\gamma$ satisfies that, when $\varepsilon$ and $\gamma$ are replaced with $\varepsilon'=\varepsilon-2\sigma$ and $\gamma'=\gamma/(1-2\sigma)$, respectively, the conditions of Lemma \ref{lem-cover}. 
	Then, by Lemma \ref{lem-absorb}, we can get a vertex set $U \subseteq V(D)$ such that $|U|\le\sigma n$ and $U$ satisfies the results of Lemma \ref{lem-absorb}. 
	Let $V'=V(D)\setminus U$, $n'=|V'|$ and $D'=D[V']$.
	Since  $\delta(D) \geq (1+\varepsilon)n$, $\alpha(D) \leq \gamma n$ and $|U|\le\sigma n$, there are $\delta(D') \geq (1+\varepsilon')n'$ and $\alpha(D') \leq \gamma n \leq \gamma' n'$.
	By Lemma \ref{lem-cover}, we have a $T_3$-tiling $\mathcal{T}_1$ such that $|V'\setminus V(\mathcal{T}_1)| \leq 12/\varepsilon'+1$. Recalling that $n$, $|U|$ and $|V(\mathcal{T}_1)|$ are divisible by 3,  $|V'\setminus V(\mathcal{T}_1)|$ is divisible by 3.
	Then, by Lemma \ref{lem-absorb}, there exists a perfect $T_3$-tiling $\mathcal{T}_2$ of $D[(V'\setminus V(\mathcal{T}_1)) \cup U]$. So $\mathcal{T}_1 \cup \mathcal{T}_2$ is a perfect $T_3$-tiling of $D$.
\end{proof}
\section{Remarks}
	As we mentioned earlier, Knierim and Su \cite{knier21} resolved Question \ref{q1.1} for all $K_r$. For digraphs, Czygrinow, Debiasio, Molla and Treglown \cite{CzygrinowDMT} showed that, for a digraph $D$ with $n \geq n_0$ vertices where $r$ divides $n$, if $\delta(D)\geq 2(1-1/r)n-1,$ then $D$ contains a perfect $\mathbb{T}_r$-tiling for every tournament $\mathbb{T}_r$ on $r$ vertices, where $r\ge 4$. When $r=3$, a more stringent degree condition is required for the above conclusion to hold, see \cite{wang}.
	Therefore, it is natural to ask the following question. 
	\begin{question}\label{q1.2}
		Let $D$ be an $n$-vertex digraph with $\alpha(D)=o(1)n$. What is the minimum degree condition on $D$ that guarantees a perfect $\mathbb{T}_r$-tiling in $D$ for $r \geq 3$?
	\end{question}
	In this paper, we present results addressing this question in the context of $T_3$. For another type of question concerning $C_3$-tilings with $r=3$, imposing condition of the independence number does not lead to an improved lower bound for the degree condition, see Wang \cite{wang}.
	
	There also exist some generalizations of independence number in digraphs.
	For a digraph $D$, a set $U \subseteq V(D)$ is called \emph{$2$-cycle independent} if  there is no directed $2$-cycle in $D[U]$. The \emph{$2$-cycle independence number} $\alpha_2(D)$ is the maximum cardinality of a $2$-cycle independent set. Clearly, $\alpha_2(D)\ge \alpha(D)$.
	A natural question arises: if we replace $\alpha(D)$ with $\alpha_2(D)$ in the Question~\ref{q1.2}, can this further relax the requirement of the degree condition for $r \geq 4$?
	
	In digraphs, the condition of the minimum semi-degree threshold is also worth considering. 
	In 2015, Treglown \cite{Treglown} showed that, for a digraph $D$ with $n \geq n_0$ vertices where $r$ divides $n$, if $\delta^0(D)\geq (1-1/r)n,$ then $D$ contains a perfect $\mathbb{T}_r$-tiling for every tournament $\mathbb{T}_r$ on $r$ vertices, where $r\ge 3$.  
	Often, a semi-degree condition will yield a better bound, so the following question can be considered:
	\begin{question}\label{que-end2}
		Let $D$ be an $n$-vertex digraph with $\alpha(D)=o(1)n$ (or $\alpha_2(D)=o(1)n$). What is the minimum semi-degree $\delta^0(D)$ condition on $D$ that guarantees a perfect $\mathbb{T}_r$-tiling in $D$ for $r \geq 3$?
	\end{question}
	Recently, Molla and Treglown \cite{TreglownEtAl2025} presented some results regarding $H$-tiling where $H$ refers to some digraphs with fewer arcs than $\mathbb{T}_r$. 
	This can sometimes enable the achievement of a perfect tiling under much smaller conditions of degree.
	\begin{question}\label{que-end3}
		Let $D$ be an $n$-vertex digraph with $\alpha(D)=o(1)n$ (or $\alpha_2(D)=o(1)n$) and $H$ be a digraph such that $|H|\mid n$. What is the minimum semi-degree $\delta^0(D)$ (or minimum degree $\delta(D)$) condition on $D$ that guarantees a perfect $H$-tiling in $D$?
	\end{question}

\end{document}